\documentclass[11pt]{amsart}
\usepackage{amsmath,amsthm,amsfonts,amscd,amssymb,eucal,latexsym,mathrsfs,stmaryrd,enumitem,bm}
\usepackage[all]{xy}
\usepackage{mathtools}
\usepackage{enumitem}  
\usepackage{comment}

\newtheorem{theorem}{Theorem}[section]

\theoremstyle{definition}
\newtheorem{definition}[theorem]{Definition}
\newtheorem{remark}[theorem]{Remark}

\newtheorem*{conjecture}{Conjecture}

\theoremstyle{plain}
\newcounter{theoremintro}
\newtheorem{theoremi}[theoremintro]{Theorem}
\newtheorem{corollaryi}[theoremintro]{Corollary}

\newcommand{\Zb}{{\mathbb Z}}

\newcommand{\eps}{\varepsilon}

\newcommand{\abs}[1]{\left|#1\right|}
\newcommand{\brak}[1]{\left(#1\right)}
\DeclarePairedDelimiter\ceil{\lceil}{\rceil}
\DeclarePairedDelimiter\floor{\lfloor}{\rfloor}
\newcommand{\sqbr}[1]{\left[#1\right]}

\numberwithin{equation}{section}

\allowdisplaybreaks

\begin{document}

\title{Group extensions preserve almost finiteness}

\author{Petr Naryshkin}
\address{Petr Naryshkin,
Mathematisches Institut,
WWU M{\"u}nster, 
Einsteinstr.\ 62, 
48149 M{\"u}nster, Germany}
\email{pnaryshk@uni-muenster.de}

%\date{April 28, 2021}

\begin{abstract}
We show that a free action $G \curvearrowright X$ is almost finite if its restriction to some infinite normal subgroup of $G$ is almost finite. Consider the class of groups which contains all infinite groups of locally subexponential growth and is closed under taking direct limits and extensions on the right by any amenable group. It follows that all free actions of a group from this class on finite-dimensional spaces are almost finite and therefore that minimal such actions give rise to classifiable crossed products. In particular, that gives a much easier proof for the recent result of Kerr and the author on elementary amenable groups.
\end{abstract}

\date{\today}

\maketitle

\section{Introduction}

When it comes to actions of discrete amenable groups, the Ornstein-Weiss tiling machinery is one of the most useful tools. Being able to approximate the space with a finite collection of Rokhlin towers provides a lot of information on the structure of the action. In the context of ergodic theory some of the applications are to the theory of orbit equivalence (originally these techniques were used by Ornstein and Weiss to extend Dye's theorem to all amenable groups) and the structure of the von Neumann crossed products. However, the classic tiling technique allows one only to approximate most of the space. In the context of topological dynamics, this is often not enough: instead, one would like to find an \emph{exact} partition of the space. In the Borel setting, this is tightly linked to the question of hyperfinite equivalence relations (see \cite{ConJacMarSewTuc20}). In the zero-dimensional setting, this leads to almost finiteness.

Suppose $G \curvearrowright X$ is an action of a discrete countable amenable group on a zero-dimensional compact metrizable space. We say that it is \emph{almost finite} if for every finite subset $K$ of the group and every $\eps>0$ there is a finite collection $\{(S_i, V_i)\}_{i=1}^n$ such that $V_i$ are clopen subsets of $X$, $S_i$ are $(K, \eps)$-invariant subsets (meaning that $\abs{S_i \Delta KS_i} < \eps \abs{S_i}$) of the group and 
$$X = \bigsqcup_{i=1}^n\bigsqcup_{g \in S_i}gV_i.$$
This was originally introduced by Matui \cite{Mat12} (in fact, for a more general case of groupoids) and later extended in an appropriate way by Kerr \cite{Ker20} to actions on spaces of arbitrary (including infinite) dimension (see Definition \ref{AFdef}).

Currently, the main application of this notion is to the classification theory of C$^*$-algebras. Specifically, to every topological dynamical system $G \curvearrowright X$ one can associate a C$^*$-algebra $C(X) \rtimes G$, which is called the crossed product of the system. The algebraic invariants of such C$^*$-algebras have been of interest for a long time (starting with the works such as \cite{Put90} and \cite{EllEva93}) and with the advent of the Elliott classification program, much attention (see, for instance, \cite{LinPhi10}, \cite{TomWin13} and \cite{Sza15}) has been devoted to understanding which crossed products fit into its scope (``are classifiable''). In the same paper \cite{Ker20}, Kerr proved that under some natural assumptions almost finiteness is a sufficient criterion for the crossed product to be classifiable. This allowed for the treatment of many systems which were inaccessible with the prior methods. In fact, it is believed that the following conjecture is true.

\begin{conjecture}
Every free action of a countably infinite discrete amenable group on a finite-dimensional compact metrizable space is almost finite.
\end{conjecture}

By now this has been verified for a large class of amenable groups. Kerr and Szabo showed in \cite{KerSza20} (based on \cite{DowZha17}) that the conjecture holds for groups of locally subexponential growth. Later it was established for groups which admit normal series of a certain form in \cite{ConJacMarSewTuc20}, and, more recently, for all elementary amenable groups in \cite{KerNar21}. The proof of the latter goes as follows. It is clear from the definition that the class of groups satisfying the conjecture is closed under taking direct limits and it is also not hard to show that it is closed under taking extensions on the right by finite groups. The bulk of work goes into showing that this class is also closed under extensions on the right by $\Zb$. Finally, by the result of Osin \cite{Osi02}, every elementary amenable group can be obtained from the trivial group using these three operation, which concludes the proof.

The main result of this paper is the following.
\begin{theoremi}
\label{MainThm}
Suppose $G \curvearrowright X$ is a free action of a countable discrete amenable group on a compact metrizable space, $H \lhd G$ is an infinite normal subgroup and the restricted action $H \curvearrowright X$ is almost finite. Then $G \curvearrowright X$ is almost finite.
\end{theoremi}
Note that in particular this covers both extensions on the right by finite groups and by $\Zb$ (unless $H$ is finite, in which case the conclusion follows from \cite{KerSza20}) and therefore gives a simple new proof of the result in \cite{KerNar21}. However, it also immediately provides a striking class of examples not covered by the prior methods. 

\begin{corollaryi}
\label{GxZ}
Let $G$ be an amenable group. Then every free action of the group $G \times \Zb$ on a finite-dimensional space is almost finite.
\end{corollaryi}

More generally, we can describe the largest class of groups that are known to satisfy the conjecture as follows.
\begin{corollaryi}
\label{MainCor}
Consider the smallest class of groups which contains all infinite groups of locally subexponential growth and is closed under taking countable direct limits and extensions on the right by an amenable group. Then every free action of a group from this class on a finite-dimensional space is almost finite.
\end{corollaryi} 

\medskip

\noindent{\it Acknowledgements.}
The author is grateful to Robin Tucker-Drob for pointing out Remark \ref{finite-normal-subg-remark} and to the anonymous reviewer for helpful comments. The research was partially funded by the Deutsche Forschungsgemeinschaft (DFG, German Research Foundation) under Germany’s Excellence Strategy – EXC 2044 – 390685587, Mathematics Münster – Dynamics – Geometry – Structure; the Deutsche Forschungsgemeinschaft (DFG, German Research Foundation) – Project-ID 427320536 – SFB 1442, and ERC Advanced Grant 834267 - AMAREC.

\section{Preliminaries}
It would be convenient for us to introduce the following framework.
\begin{definition}
Let $G \curvearrowright X$ be an action of a countable discrete group on a compact metrizable space. The \emph{type semigroup} is the abelian semigroup generated by the open sets $\{\sqbr{U} \colon U \subset X \, \mbox{open}\}$ subject to the relations 
\begin{equation}
\label{semigroup-relations}
    \sqbr{U \sqcup V} = \sqbr{U} + \sqbr{V} \quad \mbox{and} \quad \sqbr{gU} = \sqbr{U}
\end{equation}
for every $g \in G$. 

For open sets $U$ and $V$ we say that $U$ is \emph{dynamically subequivalent} to $V$ and write $U \precsim_G V$ if for every compact subset $K \subset U$ there's an open cover $K \subset \bigcup_{i=1}^nU_i$ and elements $g_1, g_2, \ldots, g_n \in G$ such that $\{g_iU_i\}_{i=1}^n$ are pairwise disjoint subsets of $V$. It is easily seen that $\precsim_G$ is transitive and respects the relations \eqref{semigroup-relations}. Thus, it extends to a partial order on the type semigroup which we again denote by $\precsim_G$.

We say that the action $G \curvearrowright X$ \emph{has comparison} if whenever $U$ and $V$ are open sets such that $\mu(U) < \mu(V)$ for every $G$-invariant measure $\mu$, then
$$U \precsim_G V.$$
Similarly, we'll say that the action \emph{has comparison on multisets} if whenever $\sum_{i=1}^n \sqbr{U_i}$ and $\sum_{j=1}^m \sqbr{V_j}$ are elements in the type semigroup such that $\sum_{i=1}^n \mu(U_i) < \sum_{j=1}^m \mu(V_j)$ for every $G$-invariant measure $\mu$, then we have
$$\sum_{i=1}^n \sqbr{U_i} \precsim_G \sum_{j=1}^m \sqbr{V_j}.$$
\end{definition}

We recall the general definition of almost finiteness, due to Kerr.

\begin{definition}[{\cite[Definition 8.2]{Ker20}}]\label{AFdef}
    Let $G \curvearrowright X$ be an action of a countable discrete amenable group on a compact metrizable space. Fix a (not necessarily invariant) compatible metric on $X$. We say that the action is \emph{almost finite}, if for every finite subset $K \subset G$ and every $\eps > 0$ there exist a finite collection $\{(S_i, V_i)\}_{i=1}^n$ (called a \emph{castle}), where
    \begin{enumerate}[label=(\arabic*)]
    \item for every $i = 1, 2, \ldots, n$ the set $S_i \subset G$ (called \emph{shape}) is finite and $(K, \eps)$-invariant, that is $\abs{S_i \Delta KS_i} < \eps \abs{S_i}$,
    \item for every $i = 1, 2, \ldots, n$ the set $V_i \subset X$ (called \emph{base}) is open,
    \item the open sets $sV_i$ for all possible choices of $i=1,\ldots, n$ and $s \in S_i$ are pairwise disjoint and have diameter smaller than $\eps$,
    \item there are subsets $S_i^\prime \subset S_i$ with $\abs{S_i^\prime} < \eps \abs{S_i}$ such that
    \[
    X \setminus \bigsqcup_{i=1}^n S_iV_i \precsim_G \bigsqcup_{i=1}^n S_i^\prime V_i.
     \]
    \end{enumerate}
\end{definition}

The connection between almost finiteness and comparison has been studied by Kerr and Szab{\'o} and it's not hard to check that the same results hold for comparison on multisets, which we record as follows.

\begin{remark}
If the action $G \curvearrowright X$ is almost finite then it has comparison on multisets. In particular, under the small boundary property (meaning that there's a base for the topology on $X$ with the boundaries being $\mu$-null for every invariant measure $\mu$) comparison, comparison on multisets, and almost finiteness are all equivalent. The proofs are exactly the same as in \cite[Theorem 9.2]{Ker20} and \cite[Theorem A]{KerSza20}.
\end{remark}

\section{Main results}

We are now ready to prove the main result of this paper.

\begin{proof}[Proof of Theorem \ref{MainThm}]
Since $H \curvearrowright X$ is almost finite it has the small boundary property (\cite[Theorem A]{KerSza20}). That implies that $G \curvearrowright X$ also has the small boundary property. Thus, it is sufficient to prove that $G \curvearrowright X$ has comparison. Let $A$ and $B$ be a open sets such that
\begin{equation}
\label{measure-ineq}
\mu(A) < \mu(B)
\end{equation}
for every $G$-invariant measure $\mu$. It follows that there's a sufficiently big F{\o}lner set $F \subset G$ and $\gamma > 0$ such that
\begin{equation}
    \label{averaged-measure-ineq}
    \frac{1}{\abs{F}}\sum_{g \in F}\nu(gA) + \gamma < \frac{1}{\abs{F}}\sum_{g \in F}\nu(gB)
\end{equation}
for all probability measures $\nu \in M(X)$. In particular, that is the case for all $\nu \in M_H(X)$ (the space of all $H$-invariant probability measures on $X$). Denote by $n$ the cardinality $\abs{F}$. 

Since $H \curvearrowright X$ is almost finite and $H$ is infinite, we can find an open castle $\{(V_j, S_j)\}_{j=1}^m$ such that all the shapes $S_i$ have cardinality at least $\frac{4n}{\gamma}$ and 
\begin{equation*}
    \nu\brak{X \setminus \bigsqcup_{j=1}^m S_jV_j} < \frac{\gamma}{4}
\end{equation*}
for every $H$-invariant measure $\nu$. Additionally, by the small boundary property we may assume that $\nu(\partial V_j) = 0$, $\nu(\partial A) = 0$, and $\nu(\partial B) = 0$ for every $\nu \in M_H(X)$. It follows that we can refine each tower in the castle according to the pattern of intersections with $A$ and $B$ and that will not change the $\nu$-measure of the castle for any $H$-invariant measure $\nu$. Thus, we may further assume that every level $hV_j$, $h \in S_j$ is either fully contained in $A$ (respectively $B$) or is disjoint from it. For $j =1, 2,\ldots m$ denote by $a_j$ (respectively $b_j$) the number of levels from the tower $(V_j, S_j)$ that are contained in $A$ (respectively $B$). Denote by $R$ the remainder $X \setminus \bigsqcup_{j=1}^m S_jV_j$ and use \cite[Proposition 3.4]{KerSza20} to find an $\eps$-neighbourhood $R^\eps$ such that 
$$\nu\brak{R^\eps} < \frac{\gamma}{4}$$
for all $\nu \in M_H(X)$.

Now for the action $G \curvearrowright X$ we have 
\begin{equation*}
A \precsim_G \sum_{j=1}^m a_j\sqbr{V_j} + \sqbr{R^\eps} \precsim_G \sum_{j=1}^m n\ceil*{\frac{a_j}{n}}\sqbr{V_j} + \sqbr{R^\eps} = \sum_{j=1}^m \sum_{g \in F} \ceil*{\frac{a_j}{n}}\left[gV_j\right] + \sqbr{R^\eps}.
\end{equation*}
Moreover, as $H$ is a normal subgroup, the action $G \curvearrowright X$ induces an action $G \curvearrowright M_H(X)$ by $\nu \mapsto \nu\brak{g\cdot}$. We shall denote the latter by $g_*\nu$. Thus, for every $H$-invariant measure $\nu$,
\begin{multline*}
    \nu\brak{\sum_{j=1}^m \sum_{g \in F} \ceil*{\frac{a_j}{n}}\sqbr{gV_j} + \sqbr{R^\eps}} - \frac{1}{n}\sum_{g \in F}\nu(gA) < \sum_{j=1}^m \sum_{g \in F} \brak{\ceil*{\frac{a_j}{n}}-\frac{a_j}{n}} g_*\nu(V_j) + \gamma/4 < \gamma/2
\end{multline*}
since for every $g \in G$ and every $\nu \in M_H(X)$
$$\sum_{j=1}^m g_*\nu(V_i) \le \sup\limits_{\nu \in M_H(X)} \sum_{j=1}^m \nu(V_i) \le \frac{1}{\min\limits_{1 \le j \le m}\abs{S_j}} \le \frac{\gamma}{4n}.$$
Similarly, 
\begin{equation*}
    \sum_{j=1}^m \sum_{g \in F} \floor*{\frac{b_j}{n}}\sqbr{gV_i} \precsim_G B
\end{equation*}
and 
\begin{equation*}
    \nu\brak{\sum_{j=1}^m \sum_{g \in F} \floor*{\frac{b_j}{n}}\sqbr{gV_i}} > \frac{1}{n}\sum_{g \in F}\nu(gB) - \gamma/2.
\end{equation*}
for every $\nu \in M^H(X)$. It follows from \eqref{averaged-measure-ineq} that 
$$\nu\brak{\sum_{j=1}^m \sum_{g \in F} \ceil*{\frac{a_j}{n}}\sqbr{gV_j} + \sqbr{R^\eps}} < \nu\brak{\sum_{j=1}^m \sum_{g \in F} \floor*{\frac{b_j}{n}}\sqbr{gV_i}}$$
for every $\nu \in M^H(X)$ and since the action of $H$ has comparison of multisets we have 
$$\sum_{j=1}^m \sum_{g \in F} \ceil*{\frac{a_j}{n}}\sqbr{gV_j} + \sqbr{R^\eps} \precsim_H \sum_{j=1}^m \sum_{g \in F} \floor*{\frac{b_j}{n}}\sqbr{gV_i},$$
and therefore
$$A \precsim_G B.$$
\end{proof}

\begin{remark}
Note that under the same assumptions as in Theorem \ref{MainThm} except that the action $G \curvearrowright X$ is free the proof still yields that it has comparison.
\end{remark}

\begin{proof}[Proof of Corollaries \ref{GxZ} and \ref{MainCor}]
It is known from \cite{KerSza20} (based on the result of \cite{DowZha17}) that all free actions of infinite groups of subexponential growth on finite-dimensional spaces are almost finite. Clearly, direct limits preserve this property and from Theorem \ref{MainThm} extensions on the right do as well. In particular, this applies to the extensions of the form $\Zb \lhd G \times \Zb$.
\end{proof}

\begin{remark}
\label{finite-normal-subg-remark}
When the space $X$ is zero-dimensional every free action of a finite group is almost finite. It follows that all free actions of infinite groups which contain finite normal subgroups of arbitrary large cardinality on zero-dimensional spaces are almost finite (and thus also on finite-dimensional spaces by \cite[Theorem B]{KerSza20}). Indeed, one can proceed as in the proof of Theorem \ref{MainThm} and choose a finite normal subgroup $H$ of sufficiently large cardinality based on the sets $A$ and $B$. 
\end{remark}


\begin{thebibliography}{999}

\bibitem{ConJacMarSewTuc20}
C. Conley, S. Jackson, A. Marks, B. Seward, and R. Tucker-Drob.
Borel asymptotic dimension and hyperfinite equivalence relations.
To appear in {\it Duke Math. J.}

\bibitem{DowZha17}
T. Downarowicz and G. Zhang. 
Symbolic extensions of amenable group actions and the comparison property.
{\it Mem. Am. Math. Soc.} {\bf 281} (2023).

\bibitem{EllEva93}
G. A. Elliott and D. E. Evans.
The structure of the irrational rotation C$^*$-algebra.
{\it Ann.\ of Math.\ (2)} {\bf 138} (1993), 477--501.

\bibitem{Ker20}
D. Kerr. 
Dimension, comparison, and almost finiteness.
{\it J. Eur.\ Math.\ Soc.}\ {\bf 22} (2020), 3697--3745. 

\bibitem{KerNar21}
D. Kerr and P. Naryshkin.
Elementary amenability and almost finiteness.
arXiv:2107.05273.

\bibitem{KerSza20}
D. Kerr and G. Szab{\'o}. 
Almost finiteness and the small boundary property.
{\it Comm.\ Math.\ Phys.}\ {\bf 374} (2020), 1--31.

\bibitem{LinPhi10}
H. Lin and N. C. Phillips.
Crossed products by minimal homeomorphisms.
{\it J.\ reine angew.\ Math.\ } {\bf 641} (2010), 95--122.

\bibitem{Mat12}
H. Matui. 
Homology and topological full groups of \'etale groupoids on totally disconnected spaces. 
{\it Proc. Lond. Math. Soc. (3)} {\bf 104} (2012), 27–-56.

\bibitem{Osi02}
D. V. Osin.
Elementary classes of groups. 
{\it Math.\ Notes} {\bf 72} (2002), 75--82.

\bibitem{Put90}
I. F. Putnam.
On the topological stable rank of certain transformation group C$^*$-algebras.
{\it Ergodic Theory Dynam.\ Systems} {\bf 10} (1990), 197--207.

\bibitem{Sza15}
G. Szab{\'o}. 
The Rokhlin dimension of topological $\Zb^m$-actions. 
{\it Proc.\ Lond.\ Math.\ Soc.\ (3)} {\bf 110} (2015), 673--694. 

\bibitem{TomWin13}
A. S. Toms and W. Winter.
Minimal dynamics and K-theoretic rigidity: Elliott's conjecture.
{\it Geom. Funct. Anal.} {\bf 23} (2013), 467--481. 

\end{thebibliography}
\end{document}